\documentclass[12pt,a4paper]{article}

\usepackage{amsmath,fullpage,mathrsfs}
\usepackage{amssymb}
\usepackage{amsfonts}
\usepackage{amsthm}
\usepackage{graphicx}

\newtheorem{theorem}{Theorem}
\newtheorem{lemma}{Lemma}
\newtheorem{problem}{Problem}

\newtheorem{corollary}{Corollary}

\def\sym{\mathcal{S}}

\def\F{\mathbb{F}}
\def\C{\mathscr{C}}
\def\D{\mathscr{D}}
\def\Z{\mathbb{Z}}
\def\lcm{\mathrm{lcm}}

\renewcommand{\geq}{\geqslant}
\renewcommand{\leq}{\leqslant}
\renewcommand{\ge}{\geqslant}
\renewcommand{\le}{\leqslant}

\def\eref#1{$(\ref{#1})$}
\def\sref#1{\S$\ref{#1}$}
\def\lref#1{Lemma~$\ref{#1}$}
\def\tref#1{Theorem~$\ref{#1}$}
\def\pref#1{Problem~$\ref{#1}$}
\def\cyref#1{Corollary~$\ref{#1}$}

\begin{document}

\title{Existence results for cyclotomic orthomorphisms\footnote{Research supported by ARC grant DP150100506.}}

\author{David Fear and Ian M. Wanless\\
\small School of Mathematical Sciences\\[-0.8ex]
\small Monash University\\[-0.8ex]
\small VIC 3800 Australia\\
\small \texttt{\{david.fear,\ ian.wanless\}@monash.edu}
}

\date{}

\maketitle

\begin{abstract}
  An {\em orthomorphism} over a finite field $\F$ is a permutation
  $\theta:\F\mapsto\F$ such that the map $x\mapsto\theta(x)-x$ is also
  a permutation of $\F$.  The orthomorphism $\theta$ is {\em
    cyclotomic of index $k$} if $\theta(0)=0$ and $\theta(x)/x$ is
  constant on the cosets of a subgroup of index $k$ in the
  multiplicative group $\F^*$.  We say that $\theta$ has {\em least
    index} $k$ if it is cyclotomic of index $k$ and not of any smaller
  index. We answer an open problem due to Evans by establishing for
  which pairs $(q,k)$ there exists an orthomorphism over $\F_q$ that
  is cyclotomic of least index $k$.

  Two orthomorphisms over $\F_q$ are orthogonal if their difference is
  a permutation of $\F_q$.  For any list $[b_1,\dots,b_n]$ of indices
  we show that if $q$ is large enough then $\F_q$ has pairwise
  orthogonal orthomorphisms of least indices $b_1,\dots,b_n$. This
  provides a partial answer to another open problem due to Evans.  For
  some pairs of small indices we establish exactly which fields have
  orthogonal orthomorphisms of those indices.  We also find the number
  of linear orthomorphisms that are orthogonal to certain cyclotomic
  orthomorphisms of higher index.

\bigskip

\noindent Keywords: finite field, cyclotomic orthomorphism, orthogonal
orthomorphisms, Weil's Theorem.
\end{abstract}

\section{Introduction}

Suppose that $\F$ is a field of finite order $q\equiv1\mod k$.  An
{\em orthomorphism} over $\F$ is a permutation $\theta:\F\mapsto\F$ such
that the map $x\mapsto\theta(x)-x$ is also a permutation of
$\F$. Since the multiplicative group $\F^*$ is cyclic it has a unique
subgroup $C_{k,0}$ of index $k$. The {\em cyclotomy classes of index $k$}
are the cosets of $C_{k,0}$ in $\F^*$. The orthomorphism $\theta:\F\mapsto\F$ is 
{\em cyclotomic of index $k$} if $\theta(0)=0$ and $\theta(x)/x$ is
invariant on cyclotomy classes of index $k$.  
It is immediate from
the definition that an orthomorphism that is cyclotomic of index $k$
is also cyclotomic of index $k'$ for any $k'$ such that $k\mid k'$
and $k'\mid(q-1)$. Let $\C_k=\C_k(q)$ denote the set of
all cyclotomic orthomorphisms of index $k$ over $\F$, and define
\begin{equation}\label{e:Dk}
\D_k=\C_k\setminus\bigcup_{\ell<k}\C_\ell.
\end{equation}
Following Niederreiter and Winterhof~\cite{NW05}, whose work we build on,
we say that orthomorphisms in $\D_k$ have {\em least index} $k$.
One aim of this paper is to establish exactly which fields possess
an orthomorphism of least index $k$. In doing so we will answer 
this problem stated by Evans \cite[Problem~28]{Eva92}:

\begin{problem}\label{p:1}
  If $a\mid b$ and $b\mid(q-1)$ we know that
  $\C_a(q)\subseteq\C_b(q)$.  When do we have
  equality?  When do we have inequality?
\end{problem}

Two orthomorphisms, $\theta$ and $\theta'$ over $\F$ are
\emph{orthogonal} if $\theta-\theta'$ is a permutation of $\F$.
Orthogonal orthomorphisms are particularly useful for constructing
orthogonal Latin squares \cite{Eva92}.  Moreover, large sets of
orthogonal Latin squares can be constructed from orthogonal cyclotomic
orthomorphisms~\cite{Eva92a}.  Our second main result shows that for
any desired finite sequence of indices, if the field is large enough
it will contain a set of pairwise orthogonal cyclotomic orthomorphisms
of the desired indices. This gives a partial answer to another
problem posed by Evans \cite[Problem~27]{Eva92}:

\begin{problem}\label{p:2}
  For $p$ an odd prime, what types of orthomorphisms can be orthogonal
  to a non-linear cyclotomic orthomorphism of index $e$, where $1<e<p-1${\rm?} 
\end{problem}

The structure of the paper is as follows. In \sref{s:basics} we define
the notation that we will use throughout, and show some basic
results. In \sref{s:exist} we answer \pref{p:1} by constructing
orthomorphisms with each plausible least index (except for some very
small fields where some indices are not achievable).  In
\sref{s:orthorth} we give a partial solution to \pref{p:2} by
showing that cyclotomic orthomorphisms of different indices can be
orthogonal provided the field is large enough.  For certain small
indices of cyclotomic orthomorphisms it is possible to compute all
cases not covered by the asymptotic results, and thereby state results
that hold for all orders.  Finally, in \sref{s:conclude} we state some
open problems and directions for future research.

\section{The basics}\label{s:basics}

We will be working in a finite field $\F=\F_q$ of order $q$.
Suppose $k\mid(q-1)$.  We use $\omega_k$ to denote a
complex primitive $k$-th root of unity, and $\eta_k$ to denote a
multiplicative character of order $k$ in $\F$. For $0\le i<k$, we
define $C_{k,i}=\eta_k^{-1}(\omega_k^i)$ to be the $i$-th cyclotomy class
in $\F$. In particular, $C_{k,0}$ is the unique subgroup of index $k$ in
$\F^*$, and $C_{k,i}$ is a coset of $C_{k,0}$ for each $i$.

A \emph{cyclotomic map} $\theta(x)$
of index $k$ is a map from $\F$ to $\F$ satisfying
\begin{equation}\label{e:cyclomap}
  \theta(x)=\left\{ 
    \begin{array}{ll} 
      0 & \mbox{if } x=0, \\ 
      a_i x & \mbox{if } x\in C_{k,i}, \\ 
    \end{array} \right.
\end{equation}
where $[a_0,a_1,\dots,a_{k-1}]$ is a list of field elements, called 
{\em multipliers}, that defines $\theta$.  We write 
$\theta=[a_0,\ldots,a_{k-1}]$.
The map defined in \eref{e:cyclomap} is
a \emph{cyclotomic orthomorphism of index $k$} if it is a bijection and
\begin{equation*}
  x\mapsto \theta(x)-x=\left\{ 
    \begin{array}{ll} 
      0 & \mbox{if } x=0, \\ 
      (a_i-1) x &  \mbox{if } x\in C_{k,i}, \\ 
    \end{array} \right.
\end{equation*}
is also a bijection. From \cite[p.~41]{Eva92} we have:

\begin{lemma}\label{l:characterise} 
  A necessary and sufficient condition for
  $\theta=[a_0,\ldots,a_{k-1}]$ to be an orthomorphism is that the
  maps $C_{k,i}\mapsto a_iC_{k,i}$ and $C_{k,i}\mapsto(a_i-1)C_{k,i}$
  both permute the cyclotomy classes. Moreover, orthomorphisms
  $\theta=[a_0,\ldots,a_{k-1}]$ and $\theta'=[a'_0,\ldots,a'_{k-1}]$
  are orthogonal if and only if $C_{k,i}\mapsto(a_i-a'_i)C_{k,i}$
  permutes the cyclotomy classes.
\end{lemma}

One way to guarantee that a map $C_{k,i}\mapsto\lambda_iC_{k,i}$
permutes the cyclotomy classes is to choose all the $\lambda_i$ from
the same cyclotomy class, say $C_{k,\ell}$. In this case $C_{k,i}$ is
mapped to $C_{k,(i+\ell)\mod k}$ for each $i$, which necessarily
produces a permutation. This observation will be particularly useful 
to us when applying \lref{l:characterise} in \sref{s:orthorth}.

\medskip

All of the above notation depends on the field we are operating in,
which will usually be implicitly understood.  However, if the context
is such that we are working with more than one field at a time, we
adopt the notational convention of specifying the order of the
relevant field in the superscript, for example, $\eta_k^q$,
$C_{k,i}^q$ or $\theta=[a_0,\ldots,a_{k-1}]^q$. Strictly speaking,
the notation $\theta=[a_0,\ldots,a_{k-1}]$ also conceals a dependence
on the choice of the character $\eta_k$, since choosing a different
character might reorder the cyclotomy classes. However, this is not an
important dependence for us, since it is just an issue of relabelling
and we will never change our choice of $\eta_k$. The sets we are most
interested in, $\C_k$ and $\D_k$, do not depend on the choice of
$\eta_k$.

Next we look at the case when the list of multipliers for a cyclotomic
orthomorphism is periodic.

\begin{lemma}\label{l:periodic}
  Let $a,b$ be integers such that $a \mid b$, and suppose
  $\theta=[c_0, c_1, \ldots, c_{b-1}]\in \C_b$.  Then $\theta \in \C_a$
  if and only if $c_i=c_{i+a}$ for $0 \leq i < b-a$.
\end{lemma}

\begin{proof}
Suppose we have a character $\eta_b$ of order $b$. Then $\eta_a=(\eta_b)^{b/a}$
is a character of order $a$. Moreover, if $x\in\eta_b^{-1}(\omega_b^{i+\lambda a})$
for some integer $\lambda$ then
\[
\eta_a(x)=(\eta_b(x))^{b/a}=(\omega_b^{i+\lambda a})^{b/a}=\omega_b^{ib/a+\lambda b}
=\omega_b^{ib/a},
\]
which is independent of $\lambda$. Hence 
$C_{b,c}\cup C_{b,c+a}\cup C_{b,c+2a}\cup\cdots\cup C_{b,c+(b/a-1)a}$
is a cyclotomy class of index $a$, for each $0\le c<a$.
This result now follows from the definition of cyclotomic orthomorphisms.
\end{proof}

Orthomorphisms in $\C_1$ and $\C_2$ are called {\em linear} and {\em
  quadratic}, respectively \cite{Eva92}.  We define a cyclotomic
orthomorphism $\theta=[a_0,\dots,a_{k-1}]$ of index $k$ to be {\em
  near-linear} if $a_0\neq a_1=a_2=\cdots=a_{k-1}$. For example, all
non-linear quadratic orthomorphisms are near-linear.  The name arises
because, by \lref{l:periodic}, if the first multiplier of a
near-linear orthomorphism was changed to agree with the later
multipliers, the orthomorphism would become linear.  Another immediate
consequence of \lref{l:periodic} is this:

\begin{lemma}\label{l:nearlinDk}
  If $\theta$ is a near-linear orthomorphism of index $k$, then
  $\theta\in\D_k$.
\end{lemma}

This last result, plus the fact that they are easy to construct, is the
main reason for our interest in near-linear orthomorphisms. From
\lref{l:characterise} we have:

\begin{lemma}\label{l:nearlin}
  The cyclotomic map $\theta=[a_0,a_1,a_1,\dots,a_1]$ of index $k$ is
  a near-linear orthomorphism if and only if $a_0\ne a_1$,
  $\eta_k(a_0)=\eta_k(a_1)$ and $\eta_k(a_0-1)=\eta_k(a_1-1)$.  
\end{lemma}

In other words, to construct a near-linear orthomorphism we simply
need to find two multipliers $a_0$ and $a_1$ from the same cyclotomy
class such that $a_0-1$ and $a_1-1$ belong to the same cyclotomy class
as each other. This is a simple task
for a computer to do. There is one restriction, though, on when we
might expect to find a near-linear orthomorphism.  There are no
near-linear orthomorphisms of index $q-1$ or $(q-1)/2$ over $\F_q$, for
any $q$. This is because
a near-linear orthomorphism differs from a linear orthomorphism on a
single cyclotomy class.  Cyclotomy classes of index $q$ or $(q-1)/2$
have size $1$ or $2$ respectively.  Two different orthomorphisms must
differ on at least $3$ elements of their domain (c.f.~\cite[Thm~5]{CW10}).
Another way to see that near-linear orthomorphisms 
cannot have index $q-1$ or $(q-1)/2$ is to apply the following result
by Niederreiter and Winterhof~\cite{NW05}.

\begin{theorem}\label{t:numnearlin}
The number of near-linear orthomorphisms over $\F_q$ of a given index $k$ is
$(q-1-k)(q-1-2k)/k^2$.
\end{theorem}

We will not explicitly need the following two results, but they seem worth
recording.

\begin{lemma}
  Let $\theta=[a_1, \ldots, a_n]^F$ be a cyclotomic orthomorphism over
  a field $F$.  Let $E$ be an extension field of $F$.  
  Then $\phi=[a_1, \ldots, a_n]^E$ is an orthomorphism over $E$.
\end{lemma}

\begin{proof}
  Immediate from \lref{l:characterise}.
\end{proof}

\begin{lemma}
  Let $E$ be an extension field of a field $F$, where $|F| \geq 4$.
  Let $a_1,a_2$ be distinct elements of $F\setminus\{0,1\}$.
  Define $k=(|E|-1)/(|F|-1)$.  
  Then $\theta=[a_1, a_2,\ldots, a_2]^E$ 
  (where $a_2$ occurs $k-1$ times)
  is an orthomorphism over $E$
  of least index $k$.
\end{lemma}

\begin{proof}
The cyclotomy class $C_{k,0}$ in $E$ is
  the non-zero elements of the subfield $F$.  Since $a_1,a_2,a_1-1$
  and $a_2-1$ are all in $C_{k,0}$, \lref{l:characterise} tells us
  that $\theta$ is an orthomorphism.  Also, $\theta$ has least index
  $k$, by \lref{l:nearlinDk}.
\end{proof}

\section{Existence of cyclotomic orthomorphisms}\label{s:exist}

The aim of this section is to establish exactly which fields have 
cyclotomic orthomorphisms of a given least index. In doing so we
will completely answer \pref{p:1}.
It is easily checked that $\C_1(3)=\C_2(3)$,
$\C_1(4)=\C_3(4)$, $\C_1(5)=\C_2(5)=\C_4(5)$ and $\C_1(7)=\C_3(7)$.
We will show:

\begin{theorem}\label{t:Ans28}
Let $q$ be any prime power.
Except for the cases just listed, $\C_a(q)\neq\C_b(q)$ whenever $a,b$
are distinct integers satisfying
$a\mid b$ and $b\mid (q-1)$.
\end{theorem}

Note that to prove $\C_a\neq\C_b$ it suffices to show that $\D_b$ is
non-empty. \tref{t:numnearlin}, together with \lref{l:nearlinDk}, does
this for all $b<(q-1)/2$. This fact was noted in \cite{NW05}, leaving 
open the cases $b=(q-1)/2$ and $b=q-1$.  These
two cases will be handled respectively
by Theorems \ref{t:q-1half} and \ref{t:q-1} below.
\tref{t:Ans28} will then follow immediately.

If $q$ is an odd prime power then the cyclotomy classes $C_{(q-1)/2,i}$
each have the form $\{\pm x\}$ for some $x\in\F_q$. We next give a 
method for building cyclotomic orthomorphisms of index $(q-1)/2$.

\begin{lemma}\label{l:evanscond}
Suppose there exists an integer $h$, permutations $\rho,\tau\in\sym_h$,
a vector $\sigma\in\{\pm1\}^h$ and 
nonzero field elements $r,m_i,v_i\in\F_q^*$ 
such that
\begin{align}
\sigma(i)m_iv_i&=rv_{\rho(i)}\label{e:phi}\\
(m_i-1)v_i&=(r-1)v_{\tau(i)}\label{e:phi-1}
\end{align}
for $1\le i\le h$. Suppose also that
\begin{align}\label{e:nonz}
0&\ne (m_1-r)\prod_{i=2}^h(m_1-m_i)\prod_{1\le i<j\le h}(v_i^2-v_j^2).
\end{align}
Then there exists an orthomorphism $\phi\in\D_{(q-1)/2}$.
\end{lemma}

\begin{proof}
We first note that $m_1\ne r$ by \eref{e:nonz}, which means that $r\ne 1$
and $m_i\ne 1$ for $1\le i\le h$, by \eref{e:phi-1}.
Define a map $\phi:\F_q\mapsto\F_q$ by
\[
\phi(x)=
\begin{cases}
m_ix&\text{if $x=\pm v_i$ for some $1\le i\le h$,}\\
rx&\text{otherwise.}\\
\end{cases}
\]
Condition \eref{e:nonz} ensures that $\phi$ is a well-defined
cyclotomic map of index $(q-1)/2$. Condition \eref{e:phi} and
\eref{e:phi-1} ensure respectively that the maps $x\mapsto\phi(x)$ and
$x\mapsto\phi(x)-x$ are permutations of $\F_q$. Hence $\phi\in\C_{(q-1)/2}$.
Finally, \eref{e:nonz} ensures that $m_1$ is distinct from
$m_2,\dots,m_h$ and $r$, which implies that $\phi\in\D_{(q-1)/2}$ 
by \tref{l:periodic}.
\end{proof}

It is routine to check that each of the following examples satisfy the
hypotheses of \lref{l:evanscond}.

\bigskip\noindent
{\bf Example 1:} 
Let $q>10$ be an odd prime power.
Suppose $\F_q$ is a field containing an element $\xi$
such that $\xi^2=-3$. Let $h=3$,
$\rho=[2,3,1]$, $\tau=[3,1,2]$, $\sigma=[1,1,1]$, $r=(1+\xi)/2$ and
\begin{align*}
v_1 &= (v_3-v_3\xi+1+\xi)/2, && m_1 = \frac{1+\xi}{v_3-v_3\xi+1+\xi}, \\
v_2 &= 1, && m_2 = (1+\xi)v_3/2, \\
    &     && m_3 = \frac{2v_3-1+\xi}{2v_3}, 
\end{align*}
where $v_3$ is any field element that is not a root of the polynomial
\begin{align*}
x(x^2-1)
&
(2x+\xi+1)
(2x+\xi-1)
(\xi x+x-2)
(\xi x-x-2)
\\
&
\times(\xi x-x-1-\xi)
(\xi x-x-3-\xi)
(\xi x-3x-1-\xi).
\end{align*}
There is a choice for $v_3$ available, given that $q>10$.

\bigskip\noindent
{\bf Example 2:} 
Let $q>12$ be an odd prime power.
Suppose $\F_q$ is a field containing an element $\xi$
such that $\xi^2=-1$. Let $h=4$,
$\rho=[2,3,4,1]$, $\tau=[4,1,2,3]$, $\sigma=[1,1,1,1]$, $r=(\xi+1)/2$,
\begin{align*} 
v_1 &= (v_3-1+\xi)/\xi,    &&m_1 = \frac{\xi-1}{2(v_3-1+\xi)}, \\
v_2 &= 1,                  &&m_2 = (1+\xi)v_3/2, \\
&                          &&m_3 = \frac{2v_3-1+\xi}{2v_3}, \\
v_4 &= (v_3+v_3\xi-1)/\xi,  &&m_4 = \frac{v_3+v_3\xi-2}{2(v_3+v_3\xi-1)}, 
\end{align*}
where $v_3$ is any field element that is not a root of the polynomial
\begin{align*}
x(x^2-1)
&
(x+\xi)
(x+\xi-1)
(2x+\xi-1)
(x+2\xi-1)\\
&
\times(x+\xi x-2)
(x+\xi x-1)
(x+\xi x-1+\xi)
(2x+\xi+\xi x-2)
(2\xi x+x-1).
\end{align*}
There is a choice for $v_3$ available, given that $q>12$.

\bigskip\noindent
{\bf Example 3:} 
Let $q$ be a power of a prime $p>3$.
Suppose $\F_q$ is a field containing an element $\xi$
such that $\xi^2=-2$ (in characteristic $11$, we must be
careful to choose $\xi=3$ since $\xi=-3$ results in $v_1=-v_2$). 
Let $h=4$, $\rho=[4, 3, 2,1]$, $\tau=[3,1,4,2]$, $\sigma=[-1,-1,-1,1]$,
$r = 1-\xi$, and
\begin{align*}
v_1 &= -2,     &&m_1 = \xi/2+1, \\
v_2 &= -\xi-1, &&m_2 = \frac{\xi+2}{\xi-2}, \\
v_3 &= 1,      &&m_3 = 3, \\
v_4 &= \xi,    &&m_4 = \xi+2.
\end{align*}

\bigskip\noindent
{\bf Example 4:} 
Suppose $\F_q$ is a field of odd characteristic other than
$3$, $7$ or $17$. Further suppose that $\F_q$ contains an element $\xi$
such that $\xi^2=2$ (if $\F_q$ has characteristic $23$,
we take $\xi=-5$, since $\xi=5$ results in $v_3=v_6$). Let $h=6$,
$\rho=[5, 4, 2, 1, 6, 3]$, 
$\tau=[2, 3, 1, 6, 4, 5]$, 
$\sigma=[1,-1,-1,1,1,-1]$,
$r = \xi$,  and
\begin{align*}
v_1 &= -1+\xi, &&m_1 = 2, \\
v_2 &= 1,      &&m_2 = 6-4\xi, \\
v_3 &= -3+\xi, &&m_3 = \frac{\xi}{3-\xi}, \\
v_4 &= 4-3\xi, &&m_4 = \frac{\xi-2}{3\xi-4}, \\
v_5 &= 2-2/\xi,&&m_5 = 2+2\xi, \\
v_6 &= 2,      &&m_6 = -1+3\xi/2.
\end{align*}

\begin{theorem}\label{t:q-1half}
There exists an orthomorphism $\phi\in\D_{(q-1)/2}$ for all odd
prime powers $q\notin\{5,7\}$.
\end{theorem}

\begin{proof}
It is easy to check by exhaustion that $\D_{(q-1)/2}=\emptyset$ when
$q\in\{5,7\}$. For $q=3$ the map $x\mapsto2x$ is in $\D_{(q-1)/2}$.
For $q=9$ we use Example 1 with $\xi=0$ and $v_3=1+\sqrt{-1}$ in 
$\F_9=\Z_3[\sqrt{-1}]$. Henceforth we assume that $q\ge11$. 
For $\F_q$ of characteristic $3$ or $7$ we can use Example 1,
and for $q\equiv1\mod4$
we can use Example 2. So suppose that $q\equiv3\mod4$ where $\F_q$ has
characteristic at least $11$. Since $-1$ is
not a quadratic residue in $\F_q$,
we know that either $-2$ is a quadratic residue or
$2$ is a quadratic residue. In the former case we use Example 3 and in the
latter case we use Example 4.
\end{proof}

The only remaining case of \tref{t:Ans28} is when $b=q-1$. We say that
an orthomorphism over $\F_q$ that lies in $\D_{q-1}$ is {\em
  non-cyclotomic}, since it makes no use of the cyclotomic structure
of $\F_q$.  Also, following the terminology and notation of
Evans~\cite{Eva92}, we define a {\em translation} $T_g$ of an
orthomorphism $\theta$ to be the orthomorphism $T_g[\theta](x) =
\theta(x + g) - \theta(g)$.  In any field of prime order, translating
an orthomorphism in $\D_c$ where $1<c<q-1$ by a nonzero $g$ will
produce a non-cyclotomic orthomorphism.  However, the same is not
guaranteed in fields of composite order \cite[p.~48]{Eva92}.
Nevertheless, translation will still be a useful tool in our next
proof.

\begin{theorem}\label{t:q-1}
There are non-cyclotomic orthomorphisms over every
field $\F_q$ of order $q>5$.  
\end{theorem}

\begin{proof}
  First, we consider the case where $\F_q$ is a field of
  characteristic $2$, with $q\ge8$.  The construction for this case
  was shown to us by A.\,B.~Evans.  Pick $a \in\F_q\setminus\{0, 1\}$.
  Set $H = \{ 0, 1, a, a+1 \} \subseteq \F_q$ and choose $c\notin H$. Let
\begin{equation}\label{e:aHa}
   \theta(x) = \left\{ 
\begin{array}{ll} 
  ax + a(a+1) & \hbox{ if } x \in H + c ,\\ 
  ax & \hbox{ else.} 
\end{array} \right.
\end{equation}
Now $ax \in a(H+c) + a(a+1)$ if and only if $x \in H+c$, which implies
that $\theta$ is a permutation.  Also, $(a-1)x \in (a-1)(H+c)+a(a+1)$
if and only if $x \in H+c$, and hence $\theta$ is an orthomorphism.
Note that $\theta(0)=0$ and there are exactly four nonzero elements $e\in\F_q$
such that $\theta(e)/e\ne a$. Since $4$ is coprime to $q-1$, it follows
that $\theta$ is non-cyclotomic.

Next, we consider the case where $\F_q$ is a field of characteristic
$c > 2$.  We note that every orthomorphism is defined by a permutation
polynomial.  By \cite[p.~47]{Eva92} or \cite[Thm\,1]{NW05}, an 
orthomorphism $\theta\in\C_k$ corresponds to a polynomial of the form
\begin{equation}\label{eq:pp}
  \theta(x) = \sum^{k-1}_{i=0} a_i x^{i(q-1)/k + 1},
\end{equation}
where $a_i\in\F_q$ for $0\le i<k$.
Let $\theta$ be an orthomorphism such that $\theta\in\D_2$ 
(we know such $\theta$ exists, for example, by \tref{t:numnearlin}).  Hence,
$\theta$ is of the form $\theta(x) = a_1 x^{(q+1)/2} + a_0 x$ where $a_1\ne0$.  
In particular, $T_1[\theta]$ is defined as:
\begin{equation*}
  T_1[\theta](x) = a_1 x^{(q+1)/2} + \frac{a_1(q+1)}{2} x^{(q-1)/2} + O(x^{(q-3)/2}),
\end{equation*}
and since this is not of the form in equation $(\ref{eq:pp})$ for any
$k < q-1$, we must have $T_1[\theta]\in\D_{q-1}$.
\end{proof}

In our proof of \tref{t:q-1}, we made use of translation to
construct non-cyclotomic orthomorphisms.  However, these still have
the property that they are easily transformed into a cyclotomic
orthomorphism, and as such still possess an underlying cyclotomic
structure.  
We say that an orthomorphism $\theta$ is {\em irregular} if
$T_g[\theta]$ is non-cyclotomic for all $g\in\F_q$.  We conjecture that
irregular orthomorphisms exist over all sufficiently large fields.
However, here we only prove the much weaker statement that there
are arbitrarily large fields that have irregular orthomorphisms.

\begin{theorem}
  Let $\F_q$ be a field of order $q=2^{2k+1}$ for some positive integer
  $k$.  Then there exists an irregular orthomorphism $\theta$ over $\F_q$.
\end{theorem}

\begin{proof}
  Consider the translates $T_g[\theta]$ of the orthomorphism $\theta$ 
  constructed in \eref{e:aHa}. We will consider two cases.

\medskip\noindent
Case 1:  $g \in H + c$. \\
If $x +g \in H+c$ then
\begin{align*}
   T_g[\theta](x) = \theta(x + g) - \theta(g) 
   = a(x + g) + a(a + 1) + ag + a(a+1) 
   = ax.
\end{align*}
If $x + g \not\in H+c$ then
\begin{align*}
   T_g[\theta](x) &= \theta(x + g) - \theta(g) 
   = ax + ag - ag - a(a+1) 
   = ax + a(a+1).
\end{align*}
In this case we have exactly $3$ non-zero elements satisfying
$T_g[\theta](x) = ax$, and since $3$ is relatively prime to $q-1$
(given that $q=2^{2k+1}\equiv2\mod3$), $T_g[\theta]$ is
non-cyclotomic.

\medskip\noindent
Case 2:  $g \not\in H + c$. \\
If $x +g \in H+c$ then
\begin{align*}
   T_g[\theta](x) &= \theta(x + g) - \theta(g) 
   = a(x + g) + a(a + 1) - ag 
   = ax + a(a+1).
\end{align*}
If $x + g \not\in H+c$ then
\begin{align*}
   T_g[\theta](x) 
   =\theta(x + g) - \theta(g) 
   =ax + ag - ag 
   =ax.
\end{align*}
In this case, there are exactly $q-5$ non-zero elements $x$ satisfying
$T_g[\theta](x) = ax$.  Since $q-5$ is coprime to $q-1$ when $q$ is a
power of $2$, the translate $T_g[\theta]$ is non-cyclotomic. 

As all its translates are non-cyclotomic, $\theta$ is irregular.
\end{proof}

\section{Orthogonal cyclotomic orthomorphisms}\label{s:orthorth}

This section is devoted to the study of orthogonality between cyclotomic
orthomorphisms with given least indices. Our main result 
(\tref{t:isorth}) will give a partial answer to \pref{p:2}. 
The proof will need the following Lemma from \cite{Babai}:

\begin{lemma}\label{l:Babai}
  Suppose $k,t\ge2$ and $k\mid (q-1)$.  Let $a_1, \ldots, a_t$ be distinct
  elements of the finite field $\F_q$.  Then the number of solutions
  $x \in \F_q$ to the system of equations $\eta_k(a_i+x)=1$ where $i=1,
  \ldots, t$ is between $qk^{-t}-t \sqrt{q}$ and $qk^{-t}+t \sqrt{q}$.
\end{lemma}

Actually, closer inspection of the proof in \cite{Babai} shows that
the lower bound can be improved somewhat, which will leave us with
fewer cases to check in some of our computations.

\begin{lemma}\label{l:Babai2}
  Suppose $k,t\ge2$ and $k\mid (q-1)$.  Let $a_1, \ldots, a_t$ be distinct
  elements of the finite field $\F_q$.  Then the number of solutions
  $x \in \F_q$ to the system of equations $\eta_k(a_i+x)=1$ where $i=1,
  \ldots, t$ is at least $qk^{-t}-(t-1-t/k+k^{-t})\sqrt{q}-t/k$.
\end{lemma}

\begin{proof}
  We use similar notation and logic to \cite{Babai}.  Let $N$ denote
  the number of solutions to the given system of equations.  Let
  $\eta_k$ be a multiplicative character of order $k$. Let $\Psi^*$
  denote the set of functions
  $\psi:\{1,\dots,t\}\rightarrow\{0,\dots,k-1\}$ that are not
  everywhere zero. Let
$$S=q+\sum_{\psi\in\Psi^*}\sum_{x\in\F_q}
\eta_k\left(\prod_{i=1}^t(a_i+x)^{\psi(i)}\right).
$$
Using Weil's theorem \cite[Thm 3.7]{Babai} to bound the terms in $S$, 
but bounding each term
according to its number $d$ of distinct roots, we have
\begin{align*}
|S|\ge q-\sum_{d=1}^t\binom{t}{d}(k-1)^d(d-1)\sqrt{q}
=q-\big((t-1)k^t-tk^{t-1}+1\big)\sqrt{q}.
\end{align*}
Hence, using the relationship between $N$ and $S$ from \cite{Babai},
\begin{align*}
N\ge\frac{|S|-tk^{t-1}}{k^t}=qk^{-t}-(t-1-t/k+k^{-t})\sqrt{q}-t/k,
\end{align*}
as claimed.
\end{proof}

\begin{theorem}\label{t:isorth} 
  Let $B=b_1,\ldots,b_n$ be a finite sequence of positive integers,
  and let $c$ be a positive integer.  Define $t=2+\sum_{i=1}^nb_i$ and 
  $k=\lcm(b_1, \ldots b_n, c)$ and let $q_0$ be the smallest positive
  integer satisfying $q_0k^{-t}-(t-1-t/k+k^{-t})\sqrt{q_0}-t/k>1$.  
  Let $\F_q$ be a
  field such that $k\mid(q-1)$ and $q\ge q_0$.  Given any set of orthomorphisms
  $T=\{\theta_1, \ldots, \theta_n\}$ over $\F_q$, where $\theta_i\in\C_{b_i}$ 
  for $1\le i\le n$, there exists a near-linear
  orthomorphism, $\phi \in \D_c$, which is orthogonal to each
  of the orthomorphisms in $T$.
\end{theorem}

\begin{proof}
Since $q\ge q_0$ and $t-1-t/k+k^{-t}\ge 0$, we have
\begin{equation}\label{e:qbigenough}
qk^{-t}-\big(t-1-\frac{t}{k}+k^{-t}\big)\sqrt{q}-\frac{t}{k}\ge
\Big(q_0k^{-t}-\big(t-1-\frac{t}{k}+k^{-t}\big)\sqrt{q_0}-\frac{t}{k}\Big)\frac{q}{q_0}>1.
\end{equation}

  Suppose that $\theta_i=[a_{i,0},a_{i,1},\ldots,a_{i,b_i-1}]$. Define
  $A_i=\{-a_{i,0},-a_{i,1},\ldots,-a_{i,b_i-1}\}$ and
  $A=\cup_{i=1}^nA_i\cup \{0,-1\}$.  Consider the system of equations
  $\eta_k(x+a)=1$ for each $a\in A$.  By \eref{e:qbigenough} and
  \lref{l:Babai2}, we can find
  two distinct solutions $x=x_1$ and $x=x_2$.  Define 
  $\phi=[x_1,x_2,\ldots,x_2]$.  By choice,
  $\eta_k(x_1)=\eta_k(x_2)=1=\eta_k(x_1-1)=\eta_k(x_2-1)$, so
  $\phi\in\D_c$, by \lref{l:nearlinDk} and
  \lref{l:nearlin}. Similarly, since
  $\eta_k(x_1-a_{i,j})=\eta_k(x_2-a_{i,j})=1$ for $1\le i\le n$ and
  $0\le j<b_i$, we see by \lref{l:characterise} that $\phi$ is
  orthogonal to each $\theta_i\in T$.
\end{proof}

Repeatedly applying the method in the proof of \tref{t:isorth},
and noting that the set $A$ grows by only $2$ elements at each
iteration, we get:

\begin{corollary}\label{cy:orthnearlin}
  Let $B=b_1,b_2,\ldots,b_n$ be a finite sequence of positive integers.
  Define $t=2n$ and $k=\lcm(b_1,\ldots,b_{n})$.  Let $q_0$ be the smallest
  positive integer satisfying the inequality 
  $q_0k^{-t}-(t-1-t/k+k^{-t})\sqrt{q_0}-t/k>1$.  
  Let $\F_q$ be a field such that $k\mid(q-1)$ and $q\ge q_0$.  Then
  there exists a set of mutually orthogonal orthomorphisms
  $\{\theta_1,\ldots,\theta_n\}$ over $\F_q$, where
  $\theta_i\in\D_{b_i}$ for $1\le i\le n$.
\end{corollary}

This gives a partial answer to \pref{p:2} (in fact, we do not use the
assumption in \pref{p:2} that we are working in a field of odd prime
order). The answer is that any conceivable set of orthogonalities is
achievable, provided the field is large enough.  The measure of
``large enough'' we have given is likely to be quite conservative in
the sense that orthogonal orthomorphisms will often exist for much
smaller orders.  For example, [8,31], [14,44,44], [47,11,11,11,11] are
orthogonal near-linear orthomorphisms of indices $2,3,5$ over
$\F_{61}$.  Also [165,121], [111,326,326], [90,132,132,132,132],
[47,175,175,175,175,175,175], are orthogonal near-linear
orthomorphisms of indices $2,3,5,7$ over $\F_{421}$. Nonetheless, some
small fields do not possess orthogonal orthomorphisms of all the
indices they plausibly might, so asymptotic results are probably the
best we can hope for in general.  The main opportunity for improving
the scope of our results might be to prove existence of sets of
orthogonal cyclotomic orthomorphisms, at least one of which has an
index comparable in size to the order of the field.  We can say little
about such sets at present.

A similar result to \tref{t:isorth} holds if we impose more structure
on our orthomorphisms.  An orthomorphism $\theta$ is called a {\em
  strong orthomorphism} (also called a {\em strong complete mapping})
if it is an orthomorphism over a field $\F$
and if $x\mapsto\theta(x) + x$ is also a permutation of $\F$.
We say that a strong orthomorphism is a {\em cyclotomic strong
  orthomorphism} of index $k$ if it is both a strong orthomorphism,
and a cyclotomic orthomorphism of index $k$.  For a survey of results
on strong orthomorphisms see \cite{Eva13} and for a proof of the
existence of cyclotomic strong orthomorphisms of index 2, see
\cite{Bel13}.  We can easily adapt \tref{t:isorth} to prove the
asymptotic existence of cyclotomic strong orthomorphisms of any given
least index.

\begin{theorem}\label{t:isstrongorth} 
  Let $B=b_1, \ldots, b_n$ be a finite sequence of positive integers,
  and let $c$ be a positive integer.  Define $t=3+\sum_{i=1}^nb_i$ and 
  $k=\lcm(b_1, \ldots b_n, c)$ and let $q_0$ be the smallest
  positive integer such that $q_0k^{-t}-(t-1-t/k+k^{-t})\sqrt{q_0}-t/k>1$.  
  Let $\F_q$ be a field such that $k\mid(q-1)$ and $q\ge q_0$.  Given
  any set of orthomorphisms $T=\{\theta_1, \ldots, \theta_n\}$ over
  $\F_q$, where $\theta_i\in\C_{b_i}$ for $1\le i\le n$, there exists
  a near-linear strong orthomorphism, $\phi \in \D_c$, which is
  orthogonal to each of the orthomorphisms in $T$.
\end{theorem}

\begin{proof}
  The proof is the same as for \tref{t:isorth}, except we put
  $A=\cup_{i=1}^nA_i\cup \{0,-1, 1\}$.
\end{proof}

\begin{corollary}\label{cy:orthnearlinstrong}
  Let $B=b_1,b_2,\ldots,b_n$ be a finite sequence of positive integers.
  Define $t=2n+1$ and $k=\lcm(b_1,\ldots,b_{n})$.  Let $q_0$ be the smallest
  positive integer satisfying the inequality 
  $q_0k^{-t}-(t-1-t/k+k^{-t})\sqrt{q_0}-t/k>1$.  
  Let $\F_q$ be a field such that $k\mid(q-1)$ and $q\ge q_0$.  Then
  there exists a set of mutually orthogonal strong orthomorphisms
  $\{\theta_1,\ldots,\theta_n\}$ over $\F_q$, where
  $\theta_i\in\D_{b_i}$ for $1\le i\le n$.
\end{corollary}

The previous section was devoted to working out when $\D_b$ is
nonempty.  As a side remark, we note two alternate proofs that $\D_b$
is non-empty for $q$ large relative to $b$. The first is by applying
\tref{t:isorth} in the case where $B$ is the empty sequence. The
second is to use recent work by Bell~\cite{Bel13}. He showed that
$|\C_k|=k!^2k^{-2k}q^k(1+O(q^{-1/2}))$ for $k$ fixed, as
$q\rightarrow\infty$ with $q\equiv1\mod k$. It follows that
\begin{equation}\label{e:Dksize}
|\D_k|=|\C_k|-\sum_{\ell<k\atop\ell\mid k}|\C_\ell|=k!^2k^{-2k}q^k\big(1+O(q^{-1/2})\big).
\end{equation}
See \cite{CL97,MMW06,SW10} for some congruences satisfied by the
number of orthomorphisms of fields of prime order.  For non-trivial
bounds on the number of all orthomorphisms see
\cite{CW10,GL16,MMW06,Tar15}, although the known lower bounds
only apply to fields of prime order.  It would be of interest to find
analogous lower bounds for fields of composite order. In that
direction, we note the following direct corollary of
\cite[Thm~3]{NW05}.

\begin{theorem}\label{t:explowbnd}
  Let $q$ be a prime power such that $q\equiv1\mod k$ for some
  $k>2$. Then $\F_q$ has at least $2^{(q-1)/k}$ orthomorphisms.
\end{theorem}

\bigskip

If preferred, the simpler bound from \lref{l:Babai} can be used
instead of the bound from \lref{l:Babai2} in \tref{t:isorth},
\tref{t:isstrongorth} and their corollaries. However, if $k,t$
are given then even in the more complicated bound the value of $q_0$
is easily found by solving the quadratic equation in $\sqrt{q_0}$. By
explicit computations in fields of orders less than $q_0$ we were able
to establish:

\begin{theorem}\label{t:orthsmallnonlin}
  Let $q$ be a prime power such that $q\equiv1\mod c$ where $1<a\le b$
  and $c=\lcm(a,b)\le 6$.  Then there exists a pair of orthogonal
  orthomorphisms $\theta\in\D_a^q$ and $\theta'\in\D_b^q$ if and only
  if $q>7$ and
\[
(q,a,b)\notin\big\{(9,2,4),\,(13,2,3),\,(13,2,6)\big\}.
\]
\end{theorem}

Note that \cyref{cy:orthnearlin} proves \tref{t:orthsmallnonlin} for 
all $q\ge9154945$ (whereas using the bound from \lref{l:Babai} would
require us to consider $q$ up to $26876448$). For $q<9154945$ we
first searched for near-linear $\theta,\theta'$. In the few cases where
near-linear examples did not exist we did an exhaustive search for 
$\theta,\theta'$. The exhaustive search was quick, since it
was only ever required for $q\le19$.

We omitted the case $a=1$ from \tref{t:orthsmallnonlin} since in that
case we can obtain a stronger result. It is known (see
e.g.~\cite[p.~39]{Eva92}) that every orthomorphism over $\F_q$ of least index
2 is orthogonal to exactly $(q-7)/2$ linear orthomorphisms.  More generally,
we have:

\begin{theorem}\label{t:1orthnearlin}
  Let $q\equiv1\mod d$ where $d\ge2$. Any near-linear $\theta\in\D_d^q$
  is orthogonal to precisely $(q-3d-1)/d$ linear orthomorphisms.
\end{theorem}

\begin{proof}
  Suppose $\theta=[a,b,b,\dots,b]$, where $a\ne b$, $\eta_d(a)=\eta_d(b)$ and
  $\eta_d(a-1)=\eta_d(b-1)$.  We wish to find all $c\notin\{0,1\}$ for
  which $[c]$ is orthogonal to $\theta$. The number of such $c$ is
\begin{align*}
\big|\{c\in\F_q:\eta_d(c-a)=\eta_d(c-b)\}\big|-2
&=\left|\Big\{c\in\F_q\setminus\{a\}:
\eta_d\Big(\frac{c-b}{c-a}\Big)=1\Big\}\right|-2\\
&=(q-1)/d-3,
\end{align*}
since $(c-b)/(c-a)$ may take any value in $C_{d,0}\setminus\{1\}$, 
and for each such value there is exactly one $c$ that achieves it.
The result follows.
\end{proof}

\tref{t:1orthnearlin} does not generalise beyond the near-linear case, at
least not in an obvious fashion. For example, consider $\F_{31}$
and choose $\eta_3$ such that $3\in C_{3,1}$. Then $[3,9,2]\in\D_3$ is
orthogonal to only one linear orthomorphism, namely $[8]$, whereas
$[3,9,16]\in\D_3$ is orthogonal to five different linear
orthomorphisms.  By comparison, all near-linear orthomorphisms in
$\D_3$ are orthogonal to seven different linear orthomorphisms, by
\tref{t:1orthnearlin}.

We remark that \tref{t:1orthnearlin} is vacuous if $q<3d+1$, by
\tref{t:numnearlin}.  However, for larger $q$ it helps us to show:

\begin{theorem}\label{t:1orthd}
Let $q$ be a prime power and 
let $d$ be a positive integer dividing $q-1$.
There exists an orthomorphism $\theta \in \D_d^q$ which is orthogonal 
to some linear orthomorphism, except when
\begin{equation}\label{e:exc}
(q,d) \in\big\{ 
(2,1),\,(3,1),\,(3,2),\,(4,3),\,(5,2),\,(5,4),\,(7,2),\,(7,3),\,(7,6),\,(13,4)
\big\},
\end{equation}
and possibly when 
$d>50$
and either $q=2d+1\equiv3\mod4$ or $q=3d+1$.
\end{theorem}

\begin{proof}
Suppose $q=ed+1$. If $e\ge4$ then by \tref{t:numnearlin} there is a
near-linear orthomorphism in $\D_d^q$, and it is orthogonal to a positive
number of linear orthomorphisms by \tref{t:1orthnearlin}. Hence we may assume
that $e\le3$. 

Next we discuss the case $q=d+1$.  Assuming $q\ge11$, then by
\tref{t:numnearlin} and \tref{t:1orthnearlin} we can find an orthogonal pair
of orthomorphisms $\theta_1\in\D_1$ and $\theta_2\in\D_2$.
Let $g\in\F_q\setminus\{0\}$. As per the proof of \tref{t:q-1},
we know that $T_g[\theta_2]\in\D_{q-1}$, whereas $T_g[\theta_1]=\theta_1$.
Since translation preserves orthogonality, these two orthomorphisms 
provide the example we need.

Now suppose $13\le q=2d+1\equiv1\mod4$. In this case, we use the
orthomorphism $\phi$ from \lref{l:evanscond}, constructed using
Example~2. It is not hard to check that $\phi$ is orthogonal to
the linear orthomorphism $[(q+1)/2]$, given that 
$(m_i-1/2)v_i=(r-1/2)v_{\pi(i)}$ where $\pi=[3,4,1,2]$.

For small orders not covered by the above results, exhaustive
computations reveal the exceptions quoted in the theorem, with no
other exceptions for $d\le50$.
\end{proof}

We conjecture that \eref{e:exc} is the full list of exceptions.

\section{Concluding remarks}\label{s:conclude}

The success of Evans~\cite{Eva92a} in constructing large maximal sets
of mutually orthogonal Latin squares using orthogonal linear and
quadratic orthomorphisms motivates further study of cyclotomic
orthomorphisms and their orthogonality properties.

We have answered \pref{p:1} completely and \pref{p:2} partially.
However, several open questions were raised. In \sref{s:exist} we
formulated the conjecture that irregular orthomorphisms exist over all
large enough fields.  We also noted in passing that the exponential
lower bound on the number of orthomorphisms given in \cite{CW10}
applies only to fields of prime order. In \tref{t:explowbnd} we gave a
similar exponential lower bound for some fields of composite order.
However, these results could surely be improved.

Our results in \sref{s:orthorth} showed existence of orthomorphisms 
of fixed least indices when the field is large. Future work could
pursue similar results for orthogonality between 
orthomorphisms with relatively large least index.
It would also be of interest to resolve the cases
$q=2d+1\equiv3\mod4$, and $q=3d+1$ which were left unsolved in
\tref{t:1orthd}.

\subsection*{Acknowledgement}  
The authors are grateful to Tony Evans for helpful advice on
orthomorphisms and for providing the example in \eref{e:aHa}.  This
article was going to press when a preprint of \cite{EMM19} became
available, which provides much better estimates for numbers of
orthomorphisms than were available when we wrote the comments leading
up to \tref{t:explowbnd}.

  \let\oldthebibliography=\thebibliography
  \let\endoldthebibliography=\endthebibliography
  \renewenvironment{thebibliography}[1]{%
    \begin{oldthebibliography}{#1}%
      \setlength{\parskip}{0.4ex plus 0.1ex minus 0.1ex}%
      \setlength{\itemsep}{0.4ex plus 0.1ex minus 0.1ex}%
  }%
  {%
    \end{oldthebibliography}%
  }

\end{document}